\documentclass[11pt,leqno,amscd,amssymb,verbatim, url]{amsart}
\usepackage{amsmath,amscd}
\usepackage{hyperref}
\usepackage{mathrsfs}
\usepackage{eucal}
\let\oldsection=

\renewcommand{\subsection}[1]{\par\vspace{.18in}\noindent\addtocounter{subsection}{1}\setcounter{equation}{0}{\bf\thesubsection\hspace{9pt}#1}}

\usepackage[english]{babel}
\usepackage[latin1]{inputenc}
\usepackage{times}
\usepackage[T1]{fontenc}
\usepackage{amsthm, amsmath, amssymb}
\usepackage{tikz}

\usetikzlibrary{arrows,shapes}
\usetikzlibrary{decorations.pathmorphing}
\usetikzlibrary{calc}

\newtheorem{thm}{Theorem}[subsection]

\makeatletter\let\c@fact\c@theorem\makeatother\newtheorem{lem}[thm]{Lemma}
\newtheorem{cor}[thm]{Corollary}
\newtheorem{hyp}[thm]{Hypothesis}
\newtheorem{prop}[thm]{Proposition}
\newtheorem{assumptions}[thm]{Assumptions}
\theoremstyle{definition}
\newtheorem{example}[thm]{Example}

\newtheorem{rem}[thm]{Remark}
\newtheorem{rems}[thm]{Remarks}

\newtheorem{construction}[thm]{Construction}

\numberwithin{equation}{subsection}
\numberwithin{thm}{section}

\newcommand{\AAA}{{\mathscr A}(\mathscr S)}
\newcommand{\ES}{{\mathscr E}({\mathscr S})}
\newcommand{\AS}{{\mathscr A}({\mathscr S})}

\newcommand{\htt}{{\text{\rm ht}}}

\newcommand{\sH}{{\mathcal H}}

\newcommand{\Ext}{{\text{\rm Ext}}}

\newcommand{\Amod}{A\mbox{--mod}}

\newcommand{\Hom}{\text{\rm Hom}}

\newcommand{\op}{{\text{\rm op}}}

\newcommand{\End}{\operatorname{End}}

\newcommand{\sO}{{\mathscr{O}}}

\newcommand{\blist}{\begin{list}{\rom{(\roman{enumi})}}{\setlength
{\leftmarg in}{0em} \setlength{\itemindent}{7ex}
\setlength{\labetwo-sidedlsep}{2ex}\setlength{\listparindent}{\parindent}
\usecounter{enumi}}}
\newcommand{\elist}{\end{list}}

\def\sO{{\mathcal O}}
\def\la{{\lambda}}

\begin{document}
\dedicatory{\it We dedicate this paper to the memory of J.A. Green}
\title[Stratifying endomorphism algebras using exact categories.]{Stratifying endomorphism algebras using exact
categories}
 \author{Jie Du}
 \address{School of Mathematics and Statistics\\University of New South Wales\\ UNSW Sydney 2052
}
 \email{j.du@unsw.edu.au {\text{\rm (Du)}}}

 \author{Brian J. Parshall}
\address{Department of Mathematics \\
University of Virginia\\
Charlottesville, VA 22903} \email{bjp8w@virginia.edu {\text{\rm
(Parshall)}}}
\author{Leonard L. Scott}
\address{Department of Mathematics \\
University of Virginia\\
Charlottesville, VA 22903} \email{lls2l@virginia.edu {\text{\rm
(Scott)}}}

\thanks{This work was supported by grants from the ARC (DP120101436, Jie Du), and the Simons Foundation (\#359360, Brian Parshall; \#359383, Leonard Scott).}
\maketitle
\setcounter{tocdepth}{3}

\begin{abstract}
This paper constructs enlargements of Hecke algebras over $\mathbb Z[t,t^{-1}]$ to certain standardly stratified algebras. The latter are obtained as endomorphism algebras of modules with dual left cell module filtrations in the sense of Kazhdan-Lusztig. A novel feature of the proofs is the use of suitably chosen exact categories to avoid difficult $\Ext^1$-vanishing conditions. 
\end{abstract}

\section{Introduction}

This paper is the second in a series aimed at proving versions of a conjecture made by the authors
in 1996. The conjecture concerns the enlargement, in a framework involving Kazhdan-Lusztig cell theory, of those Hecke endomorphism algebras  which occur naturally in the
cross characteristic representation theory of finite groups of Lie type. See \cite{DPS98} for the
original version of the conjecture, and \cite{DPS15} for a reformulation.

The \cite{DPS98} conjecture is set in the
context of a Hecke algebra $\sH$ for a finite Weyl group, using the
dual left cell modules $S_\omega$, $\omega\in\Omega$, in the sense
of \cite{Lu03}. (Thus, each $S_\omega$ is a right $\sH$-module.) The
base ring (in \cite{DPS15})  is ${\mathbb Z}[t,t^{-1}]$, where $t$
is an indeterminate. One of the conjecture's implications is that
there is a faithful right $\sH$-module $T^\dagger$, filtered by
various $S_\omega$, such that the modules
$\Delta(\omega):=\Hom_\sH(S_\omega,T^\dagger)$, with
$\omega\in\Omega$, form a stratifying system (in the sense of \cite{DPS98}) for the endomorphism
algebra $A^\dagger:=\End_\sH(T^\dagger)$. Using exact category
methods, we are able to prove this statement. See Theorem  \ref{maintheorem}
below.

 A strength of the ``stratifying system'' construction is that it is well-behaved under base change, so that the resulting
 algebra $A^\dagger\otimes k$ inherits a stratification from that of $A^\dagger$ over any Noetherian commutative ring or field
 $k$  in which $t$ is specialized to an invertible element.

The endomorphism algebras $A^\dagger$ constructed here have other good properties. In particular, based changed
versions $\widetilde A^\dagger, \widetilde T^\dagger$ can be shown to satisfy the particular ``cyclotomic'' local versions of the conjecture which were treated in \cite[Theorem 5.6]{DPS15}, using results of \cite{GGOR03} on the module categories $\mathscr O$ for rational Cherednik algebras. The present paper raises the
possibility that the \cite{DPS98} conjecture can be proved directly within the global framework of ${\mathbb Z}[t,t^{-1}]$-algebras and modules, perhaps with the present $A^\dagger$, or a close variation.

The authors began developing a general theory in \cite{DPS98} for constructing the required enlarged algebras,
centered around a set of requirements contained in what we call the  ``stratification hypothesis.''  The most difficult condition to verify
in this hypothesis is an $\Ext^1$-vanishing requirement for some of the modules involved. The present paper takes
a novel approach to this problem by building new exact categories containing the relevant modules, effectively making
the $\Ext^1$-groups involved smaller and better behaved.  While there are Specht modules and analogues for all finite Weyl groups, there are
no troublesome self-extensions, or extensions in the ``wrong order,'' because of the exact structure we construct. 
As a result, many issues of ``bad characteristic'' do not arise.  

The present paper also contains new results on exact category constructions.
In particular, the main Lemma \ref{lem2.6} gives a very general construction in an abstract setting. It very quickly leads to new
exact module categories $(\mathscr A,\mathscr E)$ for algebras $B$ over Noetherian domains $\mathscr K$, when the $K$-algebra $B_K$ obtained by base change from $\mathscr K$ to its quotient field $K$ is semisimple. The underlying additive category $\mathscr A$ is the full subcategory of $B$-mod consisting of all modules which are finitely generated
and torsion-free over $\mathscr K$. The ``exact sequences'' in $\mathscr E$ are required to be exact on certain filtrations; see Construction \ref{Construction1}. Both this construction and that of Lemma \ref{lem2.6} apply to all standard axiom systems for
exact categories.  We use the Quillen axiom system \cite{Q73}, \cite{K90},\footnote{Contrary to popular beliefs, the notion of an ``exact category'' is not exactly
well-defined. There are at least three axiom systems, all quite useful. The weakest set of axioms is that of Quillen
\cite{Q73}, as reduced to a smaller set by Keller \cite{K90}. See our Appendix A. Then there is the axiom system of
Gabriel-Roiter \cite{GR97}. Keller shows in the appendix to \cite{DRSS99} that this set is equivalent to that of Quillen
after adding the additional condition that retractions have kernels. This axiom set is generally easier to use for producing
new exact sequences from others, but the retraction axiom may be hard to verify in integral settings, or simply is not
true. It is implied by the stronger, yet simpler requirement, that all idempotents split. (An idempotent $e:A\to A$ in an additive category is called split, if $e$ can be factored as
$e=\alpha\beta$, $\alpha:B\to A$ and $\beta:A\to B$, where $\beta\alpha=1_B$, i.e., $\beta$ is a retraction.)
 The latter has several
applications, including a six term ``long exact sequence'' for $\Hom$ and $\Ext^1$ in \cite{DRSS99}, and it is used by Neeman \cite{Ne90} to build derived categories. But in the context of the proof of Theorem \ref{maintheorem} below, our main
result, idempotents do not split. For a discussion of derived categories in the Quillen framework, see \cite{K96}.} which, in particular, does not require that ``idempotents split''. This generality is especially useful when using cell modules, whose direct summands may not be cell modules, and a further Construction \ref{constructionII}, exploiting this flexibility, leads to the main theorem.

\section{Stratified algebras and exact categories} 
Throughout this section, let $\mathscr K$ be a fixed Noetherian commutative ring. Often $\mathscr K$ will also be a domain. Later, in the main application to Hecke algebras, $\mathscr K$ will be the ring ${\mathbb Z}[t,t^{-1}]$ of Laurent polynomials in a
variable $t$.  A $\mathscr K$-module
$V$ is called finite if it is finitely generated as a $\mathscr K$-module.  

By a quasi-poset, we mean a (usually finite) set $\Lambda$ with a transitive and reflexive relation $\leq$. An
equivalence relation $\sim$ is defined on $\Lambda$ by putting $\lambda\sim\mu$ if and only if $\lambda\leq\mu$ and $\mu\leq\lambda$. Let $\bar\lambda$ be the equivalence class containing $\lambda\in\Lambda$. Of course, $\bar\Lambda$
inherits a poset structure.

\subsection{Stratifying systems.} We will briefly review the notion of a (strict) {\it stratifying system}\footnote{In \cite{DPS98}, these systems were called {\it strict}  stratifying systems.  In this paper, we drop the word ``strict'' and do not consider more
general systems. (The more general stratifying systems in \cite{DPS98} allowed $\bar\mu\geq\bar\lambda$ in condition (SS3).)} for a finite
$\mathscr K$-algebra $A$ and a quasi-poset $\Lambda$. Assume that $A$ is projective over $\mathscr K$.
 For $\lambda\in\Lambda$, we require a finitely generated $A$-module $\Delta(\lambda)$, projective as a
 $\mathscr K$-module,\footnote{This condition was inadvertently omitted in the discussion \cite[p. 231]{DPS15} but explicitly
 assumed in \cite[Thm. 1.1]{DPS15} and in the original definition of a stratifying system \cite[Defn. 1.2.4]{DPS98}.}
 and a finitely generated, projective $A$-module $P(\lambda)$, together with an epimorphism $P(\lambda)\twoheadrightarrow \Delta(\lambda)$.  The following
conditions are assumed to hold:

\begin{itemize}
\item[(SS1)] For $\lambda,\mu\in\Lambda$,
$$\Hom_{ A}(P(\lambda), \Delta(\mu))\not=0\implies \lambda\leq\mu.$$

\item[(SS2)] Every irreducible $A$-module $L$ is a homomorphic image of some $\Delta(\lambda)$.

\item[(SS3)] For $\lambda\in\Lambda$, the $A$-module $P(\lambda)$ has a finite filtration by $ A$-submodules
with top section $\Delta(\lambda)$ and other sections of the form $\Delta(\mu)$ with $\bar\mu>\bar\lambda$.
\end{itemize}
When these conditions all hold, the data consisting of the $\Delta(\lambda)$, $P(\lambda)$, etc. form (by definition) a  {\it  stratifying system} for the category
 $ \Amod$ of finitely generated $A$-modules. It is also clear that $\Delta(\lambda)_{\mathscr K'}, P(\lambda)_{\mathscr K'}, \dots$ is a
 stratifying system for $ A_{\mathscr K'}$-mod {\it for any base change $\mathscr K\to \mathscr K'$}, provided $\mathscr K'$ is a Noetherian commutative
ring. (Notice that condition (SS2) is redundant, if it is known that the direct sum of the projective modules in (SS3) is a
progenerator---a property preserved by base change.)

An ideal $J$ in  the $\mathscr K$-algebra $ A$ above is called a {\it stratifying ideal} provided that
the inclusion $J\hookrightarrow A$ is $\mathscr K$-split (or, equivalently, $A/J$ is $\mathscr K$-projective)
and, for $M,N\in A/J$-mod, inflation
from $A/J$ to $A$ defines an isomorphism
\begin{equation}\label{recollement}
\Ext^n_{A/J}(M,N)\overset\sim\longrightarrow\Ext^n_{ A}(M,N),\quad\forall n\geq 0\end{equation}
of Ext-groups.\footnote{In particular, the $n=1$ case implies that $J^2=J$, see \cite{CPS90}. If an ideal $J$ is known to be projective as an $A$-module, then $J^2=J$ implies \eqref{recollement}; see Appendix B.} A {\it standard stratification} of length $n$ of $A$ is a sequence $0=J_0\subsetneq J_1\subsetneq
\cdots\subsetneq J_n= A$ of stratifying ideals\footnote{
The word ``stratifying'' may be replaced by ``idempotent'', given the projectivity assumption on $J_i/J_{i-1}$. See fn. 4.
This is the more usual definition of a standard stratification \cite{DPS98}, but not our focus here.} of $A$ such that each $J_i/J_{i-1}$ is a projective
$A/J_{i-1}$-module. If $ \Amod$ has a  stratifying system with quasi-poset $\Lambda$, then it has a standard stratification of length $n=|\bar\Lambda|$; see \cite[Thm. 1.2.8]{DPS98}.

 \begin{lem}\label{firstlemma} Suppose that $A$ has a  stratifying system as above. Let $\lambda,\mu\in\Lambda$.
  Then
   $$\Ext^1_A(\Delta(\lambda),\Delta(\mu))\not=0
 \implies\lambda<\mu.$$ \end{lem}

 \begin{proof} Assume that $\lambda\not<\mu$, and let $Q(\lambda)$ be the kernel of the given epimorphism
 $P(\lambda)\twoheadrightarrow\Delta(\lambda)$. Then $\Ext^1_A(\Delta(\lambda),\Delta(\mu))$ is homomorphic
 image of $\Hom_A(Q(\lambda),\Delta(\mu))$. But  $Q(\lambda)$ has a filtration with sections of the
 form $\Delta(\tau)$ for $\bar\tau>\bar\lambda$, so that $\Hom_A(\Delta(\tau),\Delta(\mu)))=0$ since $\tau\not\leq\mu$.
 \end{proof}

 Given a finite quasi-poset $\Lambda$, a {\it height function} on $\Lambda$ is a mapping $\htt:\Lambda\to\mathbb Z$ with the
 properties that $\lambda<\mu\implies \htt(\lambda)<\htt(\mu)$ and $\bar\lambda=\bar\mu\implies\htt(\lambda)=\htt(\mu)$. Given $\lambda\in\Lambda$,  a sequence
 $\lambda=\lambda_n>\lambda_1>\cdots>\lambda_0$ is called a chain of length $n$ starting at $\lambda=\lambda_n$.
 Then the standard height function $\htt:\Lambda\to\mathbb N$ is defined by setting $\htt(\lambda)$ to be the maximal
 length of a chain beginning at $\lambda$.

 Given $A$-modules $X,Y,$ recall that the trace module ${\text{\rm trace}}_X(Y)$ of $Y$
 in $X$ is the submodule of $X$ generated by the images of all morphisms $Y\to X$.

 \begin{prop}\label{prop2.2} Suppose that $A$ has a  stratifying system as above, and let $\htt:\Lambda\to\mathbb Z$
 be a height function. Let $\lambda\in\Lambda$.  Then the $\Delta$-sections  arising from the
 filtration (SS3) of  $P(\lambda)$
 can be reordered (constructively, as in the proof below) so that, if we set
 $$ P(\lambda)_j=\text{\rm trace}_{P(\lambda)}\big(\bigoplus_{\htt(\mu)\geq j}P(\mu)\big),$$
 then $P(\lambda)_{j+1}\subseteq P(\lambda)_j$, for $j\in\mathbb Z$, and
 $$P(\lambda)_j/P(\lambda)_{j+1}$$
 is a direct sum of modules $\Delta(\mu)$ satisfying $\htt(\mu)=j$.
 \end{prop}
 \begin{proof} First, fix $j$ maximal with a section $\Delta(\mu)$ appearing in $P(\lambda)$ such that
 $\htt(\mu)=j$.  Lemma \ref{firstlemma} implies that, whenever $M$ is a module with a submodule $D\cong \Delta(\nu)$ and $M/D\cong \Delta(\mu)$, with $\mu,\nu\in\Lambda$ and $\htt(\nu)\leq \htt(\mu)$, then $M$ is the direct sum of $D$ and a submodule $E\cong \Delta(\mu)$. Of course the quotient $M/E$ is isomorphic to $D$. This interchange of $E$ with $D$ can be repeatedly applied to adjacent $\Delta$-sections in a filtration (SS3) of $P(\lambda)$ to construct a submodule $P(\lambda)(j)$, a term in a modified filtration, which is filtered by modules $\Delta(\nu)$ with $\htt(\nu)= j$, and $P(\lambda)/P(\lambda)(j)$ filtered by modules $\Delta(\nu)$ with $\htt(\nu)<j$.
 Axiom (SS1) clearly gives $P(\lambda)(j)=P(\lambda)_j$, and $P(\lambda)_{j+1}=0$. Clearly,
 $P(\lambda)_j/P(\lambda)_{j+1}$ is a direct sum as required  by the proposition. We have not used
 projectivity of $P(\lambda)$, only its filtration properties. Induction applied to the quotient module
 $P(\lambda)/P(\lambda)_j$ completes the proof.\end{proof}

 \begin{rem} The proposition above shows that the projective modules have a canonically described filtration, given any height function $\htt$. This suggests that, if $A$ is to be realized as an endomorphism algebra of a given module, that module might also reflect that filtration in a canonical way. In \S\S3,4, this is successfully approached using semisimple base change and exact categories. The latter also builds in a height function version of the vanishing condition in Lemma
 \ref{firstlemma}.

 The proposition can also be used, in conjunction with Lemma \ref{lem2.1} below, to build stratifying ideals in an algebra Morita equivalent to $A$, and then in $A$. See \cite[Lem. 1.2.7, Thm. 1.2.8]{DPS98}. We will not need to return to this
 in this paper. \end{rem}

 \begin{lem}\label{lem2.1}  Suppose $A$ has a  stratifying system as above. Then
 $$P:=\bigoplus_{\lambda\in\Lambda} P(\lambda)$$
 is a projective generator for $A$-mod.\end{lem}

 \begin{proof} Obvious from (SS2) and (SS3). \end{proof}

\subsection{Exact categories and the stratification hypothesis.} This section provides a way to construct stratifying
systems in a setting involving exact categories.  Previously, the construction was based on assuming a ``stratification hypothesis''  in \cite[Hyp. 1.2.9, Thm. 1.2.10]{DPS98}. The method
required a difficult $\Ext^1$-vanishing condition (see \cite[Thms. 2.3.9, 2.4.4]{DPS98}).  The advantage of the exact
category approach is  that
the relevant $\Ext^1$-vanishing conditions involve smaller spaces (and so are hopefully easier to make vanish).

Let $({\mathscr A},\mathscr E)$ be exact category in the sense of Quillen \cite{Q73}, as discussed in Appendix A using
axioms of Keller \cite{K90}.
In particular, $\mathscr A$ is an additive category and $\mathscr E$ is a class of sequences $X\to Y\to Z$ satisfying
certain properties.  In the hypotheses below we will assume the more explicit setup in which $\mathscr A$
is an additive full subcategory of mod--$B$ where $B$ is a  finite and projective $\mathscr K$-algebra. The sequences
$X\to Y\to Z\in\mathscr E$ are among the short exact sequences $0\to X\to Y\to Z\to 0$ in mod--$B$. Thus, $\mathscr A$
is an ``exact subcategory'' of mod--$B$. Note, however, we do {\it not} assume that all exact sequences in mod--$B$
whose object terms lie in $\mathscr A$ necessarily belong to $\mathscr E$. 

Next, we discuss the variation of the stratification hypothesis based on the notion of an exact category. First,
there are several preliminary assumptions.

Assume there is given a collection of objects $S_\lambda\in
{\mathscr A}$ indexed by the elements $\lambda$ of a finite
quasi-poset $\Lambda$. For each $\lambda\in\Lambda$, $S_\lambda$ is
a subobject of $T_\lambda\in\mathscr A$. Write
$$ T: =\bigoplus_{\lambda\in\Lambda}T_\lambda\in \mathscr A.$$
With this notation, the following statements make up a straightforward version of the  ``stratification hypothesis'' 
in an exact categorical context.

\begin{hyp}\label{hypothesis} The stratification hypothesis holds in $(\mathscr A,\mathscr E)$ provided the following statements hold.
\begin{itemize}

\item[(1)] For $\lambda\in\Lambda$, there is a fixed sequence $\nu_{\lambda,0}, \cdots, \nu_{\lambda,l(\lambda)}$
where $l(\lambda)\geq 0$, $\nu_{\lambda,0}=\lambda$, and $\nu_{\lambda,i}>\lambda$ for each $i>0$. Also, there
is an increasing filtration
$$ 0=F^{-1}_\lambda\subseteq F^0_\lambda\subseteq \cdots \subseteq F_\lambda^{l(\lambda)}=T_\lambda$$
such that each inclusion $F^{i-1}_\lambda\subseteq F_\lambda^i$ is an inflation,\footnote{In an (abstract) exact
category setting, $F^{i-1}_\lambda\subseteq F^i_\lambda$  might be taken as a notation for a monomorphism
$F^{i-1}_\lambda\to F^i_\lambda$. In the case above, we intended that the sequence $F^{i-1}_\lambda\to F^i_\lambda\to F^i_\lambda/F_\lambda^{i-1}$ belongs to $\mathscr E$.}  such that $$F_\lambda^i/F_\lambda^{i-1}\cong S_{\nu_{\lambda,i}}$$
 for
$0\leq i\leq l(\lambda)$.

\item[(2)] For $\lambda,\mu\in\Lambda$, $\Hom_{\mathscr A}(S_\mu, T_\lambda)\not=0\implies\lambda\leq \mu.$

\item[(3)] For all $\lambda\in\Lambda$, $\Ext^1_{\mathscr E}( T_\lambda/ F^i_\lambda, T)=0, \quad\forall i\geq 0$.
(See Appendix A for a definition of $\Ext^1_{\mathscr E}$.)


\end{itemize}\end{hyp}

The proof of the following result parallels the analogous result in \cite[Thm. 1.2.10]{DPS98}, using Proposition \ref{prop4.3}(a).

\begin{thm}\label{thm2.5}Let ${\mathscr A},\mathscr E,B, T$ be as above. (In particular, $\mathscr A$
is an additive full subcategory of mod--$B$.)  Assume that Hypothesis \ref{hypothesis} holds in $(\mathscr A,\mathscr E)$.
Put
$$A^+=\End_B({ T})$$
 and, for $\lambda\in\Lambda$, define $\Delta(\lambda):=\Hom_B(S_\lambda,T)\in A^+$-mod.
Assume that each $\Delta(\lambda)$ is $\mathscr K$-projective. Then $\{\Delta(\lambda)\}_{\lambda\in\Lambda}$ is a
 stratifying system for $A^+$-mod.
\end{thm}

\begin{rem}\label{Rem2.6} The main function of condition
 (3) in Hypothesis \ref{hypothesis} in proving Theorem \ref{thm2.5} is
to ensure the existence of various exact sequences when $\Hom_{\mathscr A}(-,T)$ is applied. This exactness still works and
Theorem \ref{thm2.5} still holds if $S_\lambda$ is used in place of $T_\lambda/F^i_\lambda$, at least for the exact
categories we use. For one precise formulation, see Lemma \ref{lem3.8} below.  This discussion is necessary when
 using the Quillen axiom system. In the idempotent split context studied in \cite{DRSS99},
the functor $\Ext^1_{\mathscr E}$ is half-exact in each variable; see \cite[Thm. 1.3]{DRSS99}, who quote arguments of \cite[Thm. 1.1]{BH61}. In this case, the original version of condition (3) holds as written when all the $\Ext^1_{\mathscr E}(S_\lambda,T)$ vanish. Finally, another useful modification of Hypothesis \ref{hypothesis} (1) is obtained by replacing
$S_{\nu_{\lambda, i}}$, $i\geq 1$, by the direct sum of such objects, all with $i\geq 1$. Again, Theorem \ref{thm2.5} holds
with essentially the same arguments.
\end{rem}

\section{Some constructions of exact categories} Let $(\mathscr A, \mathscr E)$ be an exact category in the sense of
Quillen \cite{Q73}; see Appendix A. Suppose that
$\mathscr C$ is a given abelian category, and let $F:{\mathscr A}
\to {\mathscr C}$ be an additive functor.  Then $F$ is called $\mathscr E$-exact (resp., left $\mathscr E$-exact) if given
any $(X\to Y\to Z)\in \mathscr E$, the sequence $0\to F(X)\to F(Y)\to F(Z)\to 0$ (resp., $0\to F(X) \to F(Y)\to F(Z)$)
is exact in $\mathscr C$.

\begin{lem}\label{lem2.6} Let $\mathscr C$ be an abelian category.\footnote{After an earlier posting of this paper, T. B\"uhler drew our attention to Exercise
5.5 in \cite{Bu10}, which he credits to M. K\"unzer. This exercise is similar to Lemma 3.1. However, while it uses a general exact category as functor target, it does require an
apparently stronger ``admissible 
kernel preserving'' hypothesis. Indeed, in the case of an abelian category functor target, the hypothesis of Exercise 5.5 implies the hypotheses of Lemma \ref{lem2.6}. However, we do not know if there is a converse implication. The conclusions are the same for both assertions.}
Also, let $({\mathscr A}, {\mathscr E}')$ be an exact category and let $F:{\mathscr A}\to \mathscr C$ be a left ${\mathscr E}'$-exact, additive functor.
 Define $\mathscr E$ to be  the class of  those $(X\to Y\to Z)\in\mathscr E'$ such that $0\to F(X)\to F(Y)\to F(Z)
\to 0$ is exact in $\mathscr C$.  Then $({\mathscr A},{\mathscr E})$ is an exact category.\end{lem}

\begin{proof} First, since $F$ is left $\mathscr E'$-exact,  $\mathscr E$ can also be described as the class of all
$(X\to Y\to Z)\in{\mathscr E'})$ such that
$F(Y)\to F(Z)$ an epimorphism in $\mathscr C$. Axioms 0, 1 in Appendix A are immediate. Consider Axiom 2 and the following commutative
diagram in $\mathscr A$
$$
\begin{CD}
X @>>> Y' @>d'>> Z' \\
@| @V{f'}VV @VVfV \\
 X @>>> Y @>d>> Z \end{CD}
$$
in which the bottom row belongs to $\mathscr E$, so that the sequence is ${\mathscr E}'$-exact (in ${\mathscr E}'$ and $F(d):F(Y)\to F(Z)$ is an epimorphism),
and the top row is the pullback of the bottom row (through the map $f$). The object $Y'$ is identified as the
kernel of the epimorphism $(-f,d):Z'\oplus Y\to Z$ in the bottom row of the commutative diagram
$$\begin{CD}
@. @.Y @>d>> Z \\
@. @. @V{\left(\begin{smallmatrix}0 \\ 1_Y\end{smallmatrix}\right)}VV @| @.  \\
@. Y' @>>> Z'\oplus Y  @>{(-f,d)}>> Z
\end{CD}$$
 The bottom row
$Y'\longrightarrow Z'\oplus Y\overset{(-f,d)}\longrightarrow Z$ is isomorphic to $Y'\overset{\left(\begin{smallmatrix}-d'\\ f'
\end{smallmatrix}\right)}\longrightarrow Z'\oplus Y\overset{(f,d)}\longrightarrow Z$, which is shown
in \cite[p. 406]{K90} to belong to $\mathscr E'$. (See also Remark \ref{RemDog}(d) in Appendix A below for an alternate argument.) Now
apply the functor $F$, and use the natural isomorphism $F(Z'\oplus Y)\cong F(Z')\oplus F(Y)$ to obtain the following
commutative diagram
$$\begin{CD}
@. @. F(Y) @>F(d)>> F(Z) @>>> 0\\
@. @. @VVV @| @.  \\
0 @>>> F(Y') @>>> F(Z')\oplus F(Y)  @>{(-F(f),F(d))}>> F(Z) @>>> 0.
\end{CD}$$
As noted above, the morphism $F(d)$ is an epimorphism. Thus, since $F$ is left exact, the bottom row is exact, and it identifies $F(Y')$ as the
the pullback in the abelian category $\mathscr C$ of $F(f)$ and $F(d)$. Since $F(d)$ is an epimorphism, so is its pullback $F(d')$.
This verifies Axiom 2.

Finally, we must check that Axiom 2$^\circ$ holds. Consider a commutative pushout diagram
\begin{equation}\label{d2}\begin{CD} 0 @>>> X@>i>> Y @>d>> Z @>>> 0\\
@. @VgVV @VhVV @| @. \\
0 @>>> X' @>i'>> Y' @>d'>> Z @>>> 0
\end{CD}\end{equation}
in which the top row belongs to ${\mathscr E}$. We must prove that $X'\to Y'\to Z$ also belongs to ${\mathscr E}$.
But the diagram (\ref{d2}) gives the following commutative diagram
\begin{equation}\label{d3}\begin{CD}
@. @.Y @>{=}>> Y \\
@. @. @VhVV @VdVV @.\\
0 @>>> X' @>i'>> Y' @>d'>> Z @>>> 0
\end{CD}.\end{equation}
After applying $F$, we get the following commutative diagram
$$\begin{CD}
@. @.F(Y) @>{=}>> F(Y) \\
@. @. @VF(h)VV @VF(d)VV @.\\
0 @>>>F(X') @>F(i')>> F(Y') @>F(d')>> F(Z) @>>> 0
\end{CD}$$
in which $F(d)$ is an epimorphism, since the top row of (\ref{d2}) belongs to ${\mathscr E}$. This implies that
$F(d')$ is an epimorphism, and, hence, the bottom row of (\ref{d3}) is exact in $\mathscr C$.  Thus, the
bottom row of (\ref{d2}) belongs to ${\mathscr E}$, and Axiom 2$^\circ$
holds, completing the proof of the lemma.
\end{proof}

We now make some assumptions which will often be in force for the rest of this paper.

\begin{assumptions}\label{assumptions1}
Let $\mathscr K$ be a fixed Noetherian integral domain with fraction field $K$. Let
$H$ be $\mathscr K$-algebra  which is finite and torsion-free as a $\mathscr K$-module. Assume that $H_K$ is semisimple. The isomorphism
classes of irreducible right $H_K$-modules are indexed by a finite set $\Lambda$. Given $\lambda\in\Lambda$, let $E_\lambda$ denote a representative  from the corresponding irreducible class. Fix a function $\htt:\Lambda\to\mathbb Z$,
taking, for convenience, non-negative values. (We call $\htt$ a height function, though there is no immediate assumption
that $\Lambda$ is a quasi-poset.)

Let mod-$H$ be the category of $\mathscr K$-finite right $H$-modules, and let mod--$H_K$ be category of finite
dimensional right $H_K$-modules. Let $\mathscr A$ be the full subcategory of mod--$H$ which consists of $\mathscr K$-torsion-free $H$-modules.
\end{assumptions}

 For $N\in{\rm{\text{mod--}}}H_K$, the height function $\htt$ induces a natural increasing (finite) filtration
 $$0=N^{-1}\subseteq N^0\subseteq \cdots\subseteq N^i\subseteq N^{i+1}\subseteq \cdots\subseteq N,$$
 defining $N^i$ to be the sum of all irreducible right $H_K$-submodules isomorphic to $E_\lambda$ with $\htt(\lambda)\leq i$. (Thus, $N^j=N$ for all $j\geq|\Lambda|$.) Then, if $M\in\mathscr A$, there is an induced filtration
$$ 0=M^{-1}\subseteq M^0\subseteq\cdots\subseteq M^i\subseteq M^{i+1}\subseteq\cdots\subseteq M$$
on $M$ defined by
setting
$$M^i=M\cap (M_K)^i, \quad i\geq 0.$$
 Observe that each $M^i\in{\mathscr A}$, as are the modules
$M/M^i$ and $M^i/M^{i-1}$. Also, $(M^i/M^{i-1})_K$ is a direct sum of $H_K$-modules $E_\lambda$ with
$\htt(\lambda)=i$.

Our goal is to show that the above data define the structure of an exact category on the additive category $\mathscr A$
of $\mathscr K$-torsion-free right $H$-modules, once an appropriate family $\mathscr E$ of conflations $X\to Y\to Z$
has been defined.

First, we require more preliminaries, including the proposition below.
Note that if $X\overset{f}\longrightarrow Y$ is a map in $\mathscr A$, then $f$ induces a map
$f_i:X^i\to Y^i$ and a map $\overline{ f_i}:X^i/X^{i-1}\to Y^i/Y^{i-1}$ for each integer $i$. In addition, if
$g:Y\to Z$ is another morphism in $\mathscr A$, then $(gf)_i=g_if_i$ and $\overline{ g_i}\overline{f_i}=
\overline {g_if_i}$ for each $i$. Finally, if $f:X\to Y$ is an inclusion $X\subseteq Y$, then
\begin{equation}\label{definition} X\cap (Y_K)^i= X\cap (X_K)^i = X^i, \quad\forall i.\end{equation}

In the following proposition, we continue to assume that Assumptions \ref{assumptions1} are in force.

\begin{prop}\label{prop2.7} Suppose $X,Y,Z\in\mathscr A$ and $0\to X\overset f\to Y\overset g\to Z\to 0$ is an exact sequence in mod--$H$. Then, for
each $i\in\mathbb Z$, the following statements hold.

(a) The sequence $0\to X^i\to Y^i\to Z^i$ is exact in mod--$H$.

(b) The sequence $0\to X^h\to Y^h\to Z^h\to 0$ is a short exact sequence in mod--$H$, for each $h\leq i$, if and only if
$$0\to X^j/X^{j-1}\to Y^j/Y^{j-1}\to Z^j/Z^{j-1}\to 0$$
is exact for each $j\leq i$.

(c) The sequence $0\to X^j/X^{j-1}\to Y^j/Y^{j-1}\to Z^j/Z^{j-1}\to 0$ is a short exact sequence for all $j\leq i$ if and only if
$Y^g/Y^{g-1}\to Z^g/Z^{g-1}$ is an epimorphism for all $g\leq i$.

\end{prop}
\begin{proof} Throughout this proof, the word ``exact'' means exact  in the usual sense in the category of right $H$-- (or possibly $H_K$--) modules.

We first prove (a). Without loss of generality, we can assume that the map $f:X\to Y$ is an inclusion of a submodule. Clearly,
each $f_i$ is an inclusion. Also, $g_if_i=(gf)_i=0$, so that the image of $f_i$ is contained in the kernel
of $g_i$. To prove the reverse inclusion, let $y\in\ker g_i$. Thus, $y\in\ker g$, so $y\in X$. But also $y\in Y^i\subseteq
(Y_K)^i$. So $y\in X\cap (Y_K)^i=X^i$, as per (\ref{definition}). This proves (a).

We next prove (b). For every integer $j$, we have a $3\times 3$ diagram
\begin{equation}\label{dia}\begin{CD}
X^{j-1} @>>> Y^{j-1} @>>> Z^{j-1}\\
@VVV @VVV @VVV\\
X^j @>>> Y^j @>>> Z^j\\
@VVV @VVV @VVV \\
X^j/X^{j-1} @>>> Y^j/Y^{j-1} @>>> Z^j/Z^{j-1}
\end{CD}
\end{equation}
in which the columns are short exact sequences. Then assume that each $0\to X^h\to Y^h\to Z^h\to 0$ is exact
for each $h\leq j$. Then the $3\times 3$ Lemma \cite[p. 49]{Mac94} implies that $0\to X^j/X^{j-1}\to Y^j/Y^{j-1}
\to Z^j/Z^{j-1}\to 0$ is exact for all $j\leq i$.

Conversely, assume that, for any $j\leq i$, the sequence $0\to X^j/X^{j-1}\to Y^j/Y^{j-1}\to Z^j/Z^{j-1}\to 0$ is exact. By induction,
we can assume that $0\to X^{i-1}\to Y^{i-1}\to Z^{i-1}\to 0$ is exact. In addition, the composition map
$X^i\to Y^i\to Z^i$ is zero. Since the top and bottom rows of (\ref{dia}) are short exact sequences, \cite[Ex. 2, p. 51]{Mac94}
implies the middle horizontal line is a short exact sequence, as required.

As for (c), the $\implies$ direction is obvious. Conversely, it is easy to see that  if the maps $Y^g/Y^{g-1}\to Z^g/Z^{g-1}$
are epimorphisms for all $g\leq i$, then each map $Y^h\to Z^h$, $h\leq i$, is an epimorphism. Now apply (a) and (b).
\end{proof}
\medskip

In the context of Proposition \ref{prop2.7}(b), it is easy to give examples where $0\to X^h\to Y^h\to Z^h\to 0$ is not a
short exact sequence.

\begin{example} \label{ex3.4}
Let $\mathscr K=\mathbb Z$, and let $H={\mathbb Z}C_2$, where
$C_2=\{1,s\}$ is the cyclic group of order 2.  Let $S_2$ be the trivial module for $H$. It is free of rank 1 over $\mathbb Z$ with basis vector $1$. Let
 $S_1$ be the sign module for $H$, also free of rank 1 with basis vector denoted $\epsilon$ (so that $s\cdot
\epsilon:=-\epsilon)$. Consider the short exact sequence $0\to X\overset\alpha\longrightarrow Y\overset{\beta}\longrightarrow
Z\to 0$ of torsion-free $H$-modules where $X=S_2$, $Y=H$, and $Z=S_1$.  Here $\alpha(1)=1+s$, and $\beta(1)=-\beta(s)=-\epsilon$. Assign $S_{2, \mathbb Q}$ height 2 and $S_{1,\mathbb Q}$ height 1, then 
$$\begin{cases}X^{1}= S_2^{1}=0;\\ Y^{1}=H^{1}= {\mathbb Z}(1-s);\\ Z^{1}=Z.\end{cases}
$$
Then $\beta(Y^{ 1})=2{\mathbb Z}\epsilon$, so that $Y^{ 1}\to Z^{1}$ is not surjective. Thus, taking $h=1$, the sequence
$0\to X^h\to Y^h\to Z^h\to 0$ is not exact. However, with the same height function, but  
 interchanging the roles of $X$ and $Z$, the short exact sequence $0\to Z\overset{\phi}\to Y\overset{\psi}\to X\to 0$ (where 
$\phi(\epsilon)=1-s$, and $\psi(1)=\psi(s)=1$) has the property that $0\to Z^h\to Y^h\to X^h\to 0$ is exact
for all $h$.  If the height function assignment is reversed, then the sequence $0\to X^h\to Y^h\to Z^h\to 0$ is also exact for all $h$. \end{example}

\begin{construction}\label{Construction1}  {\it Keep Assumptions \ref{assumptions1} with $H$, $\mathscr A$ and $\htt$ as described
there.} Now define $\mathscr E$ as follows. A pair $(\iota,\delta)$ of morphisms $X\overset{\iota}\to Y$ and $Y\overset{\delta}\to Z$ in $\mathscr A$ belongs to $\mathscr E$ if and only if the sequence
$0\to X\overset \iota\to Y\overset \delta\to Z\to0$ and the induced sequences $0\to X^i/X^{i-1}\to Y^i/Y^{i-1}\to Z^i/Z^{i-1}\to 0$ ($i\in\mathbb N$) are all exact in mod-$H$.
\end{construction}

We note that, by
Proposition \ref{prop2.7}, each sequence $0\to X^i\to Y^i\to Z^i\to 0$ is also exact. The Example \ref{ex3.4} shows that the height function determines which sequences are exact (i.e., belongs to $\mathscr E$).

\begin{thm}\label{thm2.8} The pair $(\mathscr A,\mathscr E)$  is an exact category.\end{thm}

\begin{proof} First observe that there is the standard exact category $({\mathscr A},{\mathscr E}')$. Here ${\mathscr E}'$ consists of
all exact triples $X\to Y\to Z$ in mod--$H$ with $X,Y,Z$ objects in $\mathscr A$ (i.e., $X,Y,Z$ are $\mathscr K$-torsion-free). Let $\mathscr C$ be the abelian category of right $H$-modules (not necessarily finitely generated), and $F:\mathscr A\to \mathscr C$ the functor
$FX=\bigoplus_{i\geq 0}X^i$. Then $F$ is left $\mathscr E'$-exact, and $\mathscr E$ (as defined in Construction \ref{Construction1})
consists of precisely those $(X\to Y\to Z)\in {\mathscr E}'$ for which $0\to F(X)\to F(Y)\to F(Z)\to 0$ is a short exact sequence in
$\mathscr C$. (Apply Proposition \ref{prop2.7}(b).) Thus, $({\mathscr A}, {\mathscr E})$ is an exact
category by Lemma \ref{lem2.6}.
\end{proof}

\begin{rem}Though the construction of $(\mathscr A,\mathscr E)$ requires the tools of exact category theory,
they can all be interpreted here in the larger (and more familiar) category mod--$H$.  Similar remarks
apply to the second construction below.
\end{rem}

\medskip\noindent
\begin{construction}\label{constructionII}{\it  Keep Assumptions \ref{assumptions1}.
For each integer $i$, let ${\mathscr S}_i$ be a full, additive
subcategory of $\mathscr A$ such that if $S\in {\mathscr S}_i$, then
$S_K$ is a direct sum of irreducible right $H_K$-modules
having height $i$.\footnote{We think of ${\mathscr S}_i$ as a
special class of objects in $\mathscr A$; the stated condition on
$S_K$ is necessary, but not always sufficient for membership in
${\mathscr S}_i$.} (If $i$ is not in the image of the height
function, then put ${\mathscr S}_i:=[0]$.)  Let $\mathscr S$ be the
set-theoretic union of the ${\mathscr S}_i$. Let ${\mathscr
A}({\mathscr S})$ be the full subcategory of $\mathscr A$ above
having objects $M$ satisfying  ${M}^j/{M}^{j-1}\in {\mathscr S}_j$
for all $j$ (or, equivalently, $M^j/M^{j-1}$ is in $\mathscr S$ for
all integers $j$).

Let $\mathscr E$ be as in Construction \ref{Construction1}. Define ${\mathscr E}({\mathscr S})$ to be the class of those conflations $X\to Y\to Z$ in $\mathscr E$ such that $X,Y,Z\in{\mathscr A}({\mathscr S})$ and with the additional property that, for each integer $i$,
$$0\to X^i/X^{i-1}\to Y^i/Y^{i-1}\to Z^i/Z^{i-1}\to 0$$ is a split short exact sequence
in mod--$H$. (Thus, by definition, $\mathscr E({\mathscr S})\subseteq \mathscr E$.)}\end{construction}

\begin{thm} \label{thm2.9} The pair $({\mathscr A}({\mathscr S}),{\mathscr E}({\mathscr S}))$ is an exact category.\end{thm}

\begin{proof} The first two axioms in Appendix A are easily verified. (Note again that ${\mathscr E}({\mathscr S})\subseteq\mathscr E$.)
To check Axiom 2, consider the diagram
$$\begin{CD}
X @>>> Y' @>>> Z'\\
@VVV @VVV @VfVV \\
X @>>> Y @>>> Z\end{CD}$$
where the bottom row is in ${\mathscr E}({\mathscr S})$ and the top row is a pullback (in mod--$H$) with $Z'\in {\mathscr A}({\mathscr S})$. However, since the bottom row lies in $\mathscr E$, we have that $X'\to Y'\to Z'$ also belongs to
$\mathscr E$. The issue is whether it splits section by section (which, in particular, would imply that $Y'\in{\mathscr A}({\mathscr S})$). This splitting at the section level follows easily from the fact that the pullback of a split short exact sequence is split. A similar
argument gives Axiom 2$^\circ$.\end{proof}

The following lemma shows a common vanishing condition leads to expected exact sequences.

\begin{lem} \label{lem3.8} Suppose that $X\in{\mathscr A}({\mathscr S})$ satisfies $\Ext^1_{{\mathscr E}({\mathscr S})}(S,X)=0$,
for all $S\in\mathscr S$. Let $E\to F\to G$ belong to ${\mathscr E}({\mathscr S})$. Then
$$0\to\Hom_{{\mathscr A}({\mathscr S})}(G,X)\to \Hom_{{\mathscr A}({\mathscr S})}(F,X)
\to \Hom_{{\mathscr A}({\mathscr S})}(E,X)\to 0$$
is a short exact sequence of $\mathscr K$-modules.\end{lem}

\begin{proof} The lemma is obvious, from Proposition \ref{prop4.3}(a), when $F=F^h$ for some $h\in\mathbb Z$ and $E=F^{h-1}$, since $G=F^h/F^{h-1}\in{\mathscr S}_h$.

This special case applies to all columns of the  commutative diagram, upon applying the functor $\Hom_{{\mathscr A}(\mathscr S)}(-,X)$ to the diagram
$$\begin{CD}
E^{h-1} @>>> F^{h-1} @>>> G^{h-1}\\
@VVV @VVV @VVV\\
E @>>> F @>>> G \\
@VVV @VVV @VVV\\
E^h/E^{h-1} @>>> F^h/F^{h-1} @>>> G^h/G^{h-1}.
\end{CD}
$$
Here, $h$ is chosen so that $F=F^h$, and it follows
that $E=E^h$ and $G=G^h$.
Moreover, we can assume the top row of the resulting diagram is exact by induction (on, say, the number of indices $j$ for which $F^j/F^{j-1}\not=0$). Finally, the bottom split row, of course, remains split exact in the new $3\times 3$ diagram.
Since the middle row of the latter satisfies the hypothesis of \cite[Ex. 2, p. 51]{Mac94}, it defines a short exact sequence. This proves the lemma.
\end{proof}


\section{Some further results for $({\mathscr A}({\mathscr S}),{\mathscr E}({\mathscr S}))$ and construction of $T^\dagger$}
In this section, we consider further the exact category $({\mathscr A}({\mathscr S}), \mathscr E({\mathscr S}))$
introduced in  Construction \ref{constructionII}. In particular, Assumptions 3.2 are in force.

\begin{prop}\label{propA} Let $M,N\in{\mathscr A}(\mathscr S)$, and let $h$ be any integer.

(a) There is a
natural isomorphism
$\Ext^1_{\ES}(N^h,M)\cong\Ext^1_{\ES}(N^h, M^h)$.

(b) In particular, if $S\in{\mathscr S}_h,$ we have
$$\Ext^1_{\ES}(S,M)\cong\Ext^1_{\ES}(S, M^h).$$

(c) Assume that $S\in{\mathscr S}_h$. Suppose that $M=M^h$ and $M^{h-1}=0$. Then $\Ext^1_{\ES}(S,M)=0$.
\end{prop}

\begin{proof} Without loss of generality, take $N=N^h$ in (a). Obviously, there is a natural transformation
$$\eta(N,M):\Ext^1_{\ES}(N,M)\to\Ext^1_{\ES}(N,M^h)$$
which sends $(M\to Y\to N)\in{\ES}$ to $(M^h\to Y^h\to N^h)\in{\ES}$.
 The inverse is obtained by pushout. This proves (a), and (b) follows. Finally, (c) follows immediately
 from the definition of $\mathscr E(\mathscr S)$.
\end{proof}

We also have the following result. It is immediate from the definitions.
\begin{lem}\label{critical} Let $M\in \mathscr A(\mathscr S)$. If $S\in{\mathscr S}_h$, then $\Ext^1_{\ES}(S,M^{h-1})\cong\Ext^1_H(S,M^{h-1})$.\end{lem}

\begin{prop}\label{prop4.4} Let $S\in {\mathscr S}_h$,  let $M\in{\AAA}$, and let $j$ be a non-negative integer. There is a 6-term exact sequence
$$\begin{aligned}
0\to\Hom_{\AAA}(S, M^j)\to&\Hom_{\AAA}(S, M)\to \Hom_{\AAA}(S,M/M^j) \\ \overset{f}\to &
\Ext^1_{\ES}(S, M^j)\overset{g}\to\Ext^1_{\ES}(S, M)\overset{}\to\Ext^1_{\ES}(S,M/M^j).\end{aligned}$$
It is compatible with the first 6 terms of the long exact sequence for the functor $\Hom_{\mathscr A}(S,-)=\Hom_{H}(S,-)$ applied to
the short exact sequence $0\to M^j\to M\to M/M^j\to 0$ of $H$-modules. This sequence belongs to ${\mathscr E}({\mathscr S})$.  \end{prop}

\begin{proof} The last assertion that the sequence $M^j\to M\to M/M^j$ belongs  to $\ES$ follows  from the hypothesis that $M\in {\mathscr A}(\mathscr S)$ and the definition of $(\AS,\ES)$. 
All the maps are standard: the connecting map $f$ uses pullbacks, and the other $\Ext^1_{\mathscr E(\mathscr S)}$-maps arise  from
functorality (and use pushouts). The composition of any two consecutive maps is zero.   Note that 
$\Hom_{{\mathscr A}({\mathscr S})}(S,-)=\Hom_{\mathscr A}(S,-)=\Hom_H(S,-),$ when applied 
to ${\mathscr A}({\mathscr S})$. All
$\Ext^1_{{\mathscr E}(\mathscr S)}$--groups
are contained in (and are compatible with) their $\Ext^1_H$ counterparts. Now an element in the kernel of
$g$ is also in the kernel of its classical counterpart, so lies in the image of $f$, since the first three terms of the ``long
exact sequence'' are identical to those in the classical case (i.e., mod-$H$).

Now consider exactness at the 5th term. By Proposition \ref{propA}(b), we  may assume $M=M^h$. If $j\geq h$,
then $g$ is clearly an isomorphism and exactness at the
5th term follows. If  $j<h$, then $(M^j)^{h-1}=M^j$ and so, by
Lemma \ref{critical}, $\Ext^1_{\ES}(S, M^j)=\Ext^1_{H}(S, M^j)$. Thus,  the first four terms of the ``long
exact sequence'' are identical to those in mod-$H$. Now, exactness at the next term follows as before.
\end{proof}

\begin{rem} Observe that exactness of the first 5 terms of the proposition holds for any $S\in\mathscr A({\mathscr S})$, not just in
${\mathscr S}$. Also, as noted in Proposition \ref{prop4.3}(b), the $\Ext^1_{\ES}$-groups  above are all naturally
$\mathscr K$-modules. The proof of that proposition shows they are $\mathscr K$-submodules of the corresponding
$\mathscr K$-modules $\Ext^1_H$. All maps in the above proposition are $\mathscr K$-module maps.

When ``idempotents split'', there is a general 6 term exact sequence; see fn. 1. 
For any exact category satisfying the Quillen axioms, there is always a general 4 term exact sequence; see 
Proposition A.2(a).

\end{rem}

\begin{cor}\label{prevcor}Let $S\in\mathscr S_h$, and let $M\in\mathscr A(\mathscr S)$.

(a) The map $\Ext^1_{\ES}(S, M^{h-1})\to\Ext^1_{\ES}(S,M^h)\cong\Ext^1_{\ES}(S,M)$ is surjective.

(b) We have $\Ext^1_{\ES}(S,M)=0$ if and only if the map
$$\Hom_{{\mathscr A}({\mathscr S})}(S, M^h/M^{h-1})\to\Ext^1_{\ES}(S, M^{h-1})$$
 is surjective.

(c) Suppose that $\Ext^1_{\ES}(S,M^{h-1})$ is generated as a $\mathscr K$-module by $\epsilon_1, \cdots,
\epsilon_n$. Let $M^{h-1}\to N\to S^{\oplus n}$ represent the element of 
$$\Ext^1_{\ES}(S^{\oplus n},M^{h-1})
\cong\Ext^1_H(S^{\oplus n},M^{h-1})\quad{\text{\rm (see Lemma \ref{critical})}}
$$ corresponding to $\chi:=\epsilon_1\oplus\cdots\oplus\epsilon_n$. Finally, suppose there is
a commutative diagram
$$\begin{CD}
M^{h-1} @>>> N @>>> S^{\oplus n}\\ @| @VVV @VfVV \\
M^{h-1} @>>> M^h @>>> M^h/M^{h-1}\end{CD}
$$
where $f$ is a morphism in $\AAA$. Then $\Ext^1_{\ES}(S,M)=0$.
\end{cor}

\begin{proof} Assertion (a) follows from the 6-term exact sequence of Proposition \ref{prop4.4} and the fact that $\Ext^1_{\ES}(S,
M^h/M^{h-1})=0$. (The equality
follows from Proposition \ref{propA}(b).)  The proof of (b) is similar. Next, if $M^{h-1}\to N_i\to S$ corresponds to $\epsilon_i$, there is
a pullback (of the top row in the display above) with top row $\epsilon_i$. Thus, $\epsilon_i$ is a pullback of $M^{h-1}\to M^h \to M^h/M^{h-1}$ under the
evident composite $g_i:S\to S^{\oplus n}\overset f\to M^h/M^{h-1}$. Consequently, the image of $g_i\in\Hom_{\mathscr A}(S, M^h/M^{h-1})$ under the connecting homomorphism to $\Ext^1_{\ES}(S,M^{h-1})$ is $\epsilon_i$. Since $i$ was arbitrary, the connecting homomorphism in (b) is surjective. Hence, $\Ext^1_{\mathscr E(\mathscr S)}(S, M)=0$.\end{proof}

\begin{rem}The argument above has already appeared in a module theoretic form in \cite{DPS15}. However, the argument given there required stronger hypotheses, e.g., that $\Ext^1_H(S,S)=0$.
\end{rem}

\begin{thm}\label{thm4.7} Assume that each ${\mathscr S}_i$ is strictly generated as an additive category by finitely many objects, i.e.,
every object in ${\mathscr S}_i$ is isomorphic to a finite direct sum of a given finite set of objects in ${\mathscr S}_i$.
Let $M\in \AAA$.
 Then there exists an object $X$ in $\mathscr A(\mathscr S)$ and an inflation $M\overset{i}\to X$ such that $\Ext^1_{\ES}(S,X)=0$ for all
$S\in\mathscr S$.
In addition
 if $h$ is chosen minimal such that $M^{h-1}\neq 0$, it may  be assumed that the inflation induces an isomorphism $M^{h-1}\cong X^{h-1}$.
\end{thm}

\begin{proof} Without loss of generality, we can assume that $M\not=0$, and also that $\Ext^1_{\ES}(S, M)\neq
0$ for some $S\in \mathscr S$.  Choose an integer $h$ minimal with
such a non-vanishing occurring for some $S\in \mathscr S_h$. Note that
$M^{h-1}\neq 0$ by Proposition \ref{propA}(c). We will next enlarge
$M$ to an object $X$, closer to the $X$ required in the theorem.

  Let $S_1,\cdots, S_m$ be generators for $\mathscr S_h$. For each index $i$, let $\epsilon_{i,1},\cdots,\epsilon_{i,n_i}$ be a finite
set of generators for $\Ext^1_{H}(S_i, M^{h-1})\cong\Ext^1_{\ES}(S_i,
M^{h-1})$. Form an extension
 $0\to M^{h-1}\to Y_i \to S_i^{\oplus n_i}\to 0$ corresponding to
 $\chi_i:=\epsilon_{i,1}\oplus\cdots\oplus\epsilon_{i,n_i}
\in\Ext^1_{H}(S_i^{\oplus n_i}, M^{h-1})$. Put $\chi:=\chi_1\oplus \cdots\oplus \chi_m$, and let
$\chi'\in\Ext^1_H(M^h/M^{h-1}, M^{h-1}) $ correspond to the
extension $0\to M^{h-1}\to M^h\to M^h/M^{h-1}\to 0$. Put $Z:=\oplus_i S_i^{\oplus n_i}
\oplus M^h/M^{h-1}$, and let $M^{h-1}\to X^h\to Z$ correspond to
$\chi\oplus\chi'$. Observe there is a  
commutative diagram
$$
\begin{CD}
M^{h-1}@ >>> Y_i @>>> S_i^{\oplus n_i}\\
@| @VVV @VVV\\
M^{h-1} @>>> X^h @>>> Z,\end{CD}
$$
in which the top row corresponds to $\chi_i$ and the bottom row to
$\chi\oplus\chi'$ .  
Comparison with 
Corollary \ref{prevcor}(c), allowing for the differences in notation, shows 
$\Ext^1_{\ES}(S_i,X^h)=0$ for all $i$. Thus, $\Ext^1_{\ES}(S,X^h)=0$, for all $S$ in ${\mathscr S}_h$.    Note we have
the same vanishing for $S\in{\mathscr S}_j$ with $j<h$,  by our
choice of $h$.  In all cases, we can replace $X^h$ with any $X'$
containing it with $(X')^h=X^h$.

     So far, we have not constructed an object $X$, only $X^h$. However,
the latter may be viewed as the middle term of an exact sequence of right
$H$-modules $0\to M^h\to X^h\to S'\to 0$, where $S':=\bigoplus
S_i ^{\oplus n_i}\in \mathscr S_h$.  This sequence clearly corresponds to a conflation in $\ES$.
(Note how $Z$ above is split.)  Applying a pushout construction using $M^h\to M$ (see
Propostion \ref{propA}(b) and its proof), we obtain an object $X$ in
$\mathscr A(\mathscr S)$ which contains a copy of $M$ under an
inflation, and has our constructed $X^h$ as its image under the
functor $(-)^h$. In addition $X^j=M^j$ for $j\leq h-1$.

    Applying Proposition \ref{propA}(b) again, we find that $\Ext^1_{\ES}(-,X)$
vanishes on all objects in $\mathscr S_j$ with $j\leq h-1$ (and, thus, $j\leq h$). Now repeat
the argument with $X$ in the role of $M$. This requires a bigger
$h$, unless $\Ext^1_{\ES}(S, X)$ already vanishes for all $S\in
\mathscr S$. Eventually the process stops, at which point $X$
satisfies all requirements of the theorem.
\end{proof}

\begin{rem}\label{4.8}
We do not have here  any canonical choice  for $X$. In the more local context of \cite{DPS15},
we did obtain some useful uniqueness results, effectively characterizing analogs of $X$ as injective hulls in a suitable category; see \cite[Props.~6.1\&6.2]{DPS15}.
\end{rem}

For the main result, we let $H$ be the Hecke algebra $\mathcal H$ over $
\mathscr K=\mathbb Z[t,t^{-1}]$ associated to a finite Coxeter system $(W,S)$. See
\cite[Ch. 8]{Lu03} for a very general ``unequal parameter'' version
of $\mathcal H$, and a corresponding Kazhdan-Lusztig cell theory. We use dual
left cell modules $S_\omega$ as generators for the various additive
categories ${\mathscr S}_i$. Here $\omega$ is a left cell in $W$.
There are also right cells, and two-sided cells. These are all
defined as equivalence classes associated to certain quasi-posets in
$W$. We shall make use of the opposite $\leq_{LR}^{\text{op}}$ of
the quasi-poset order $\leq_{LR}$, defined in \cite[Ch. 8]{Lu03}.
However, we view it as an order on the set $\Omega$ of left cells
(rather than on $W$). Using $\leq_{LR}^{\text{op}}$ for $\leq$ in
the discussion above Proposition \ref{prop2.2} earlier in this
paper, choose a height function $\htt: \Omega\to \mathbb{Z}$ on the quasi-poset $(\Omega,\leq_{LR}^\op)$. Thus, $\htt$ takes a constant value for left cells occuring in the same two-sided cell. Observe that, given
two left cells $\omega,\omega'$, if $S(\omega)_K$ and $S(\omega')_K$ have a common composition factor
then $\omega$ and $\omega'$ are contained in a two-sided cell and so $\htt(\omega)=\htt(\omega')$. It follows that $\htt$ defines a height function, still denoted $\htt$, on the
set $\Lambda$ of irreducible $\mathcal H_K$-modules, equipped with the same quasi-poset structure.  Now,
for each integer $i$, define ${ \mathscr S}_i$ as the additive
category generated by all dual left cell modules $S_\omega$ for
which $\htt(\omega)=i$.

For $\omega,\omega'\in\Omega$, define a preorder
$$\omega\preceq\omega'\iff \htt(\omega)<\htt(\omega'), \text{ or } \htt(\omega)=\htt(\omega')\text{ and }\omega\sim_{LR}\omega'$$
(compare  the order $\leq_f$ given on \cite[p.236]{DPS15}). Then $(\Omega,\preceq)$ becomes a quasi-poset and $\htt$ remains a height function with respect to $\preceq$. We remark that the preorder $\preceq$ is ``strictly compatible'' with the partition into two-sided cells, in the sense of \cite[Conj. 1.2]{DPS15}.

For each $\omega\in\Omega$, construct $X=X_\omega$ as in the above
theorem, with $M=S_\omega$. Choose positive integers $m_\omega$,
$\omega\in\Omega$, and let
$$T^\dagger=\bigoplus X_\omega^{\oplus m_\omega}.
$$
The use of chosen positive integers $m_\omega$ is a useful flexibility---all choices of $m_\omega>0$ lead to Morita equivalent endomorphism algebras $A^\dagger$ in the statement below.

 \begin{thm}\label{maintheorem} The $\mathbb Z[t,t^{-1}]$-algebra $A^\dagger:=\End_{\mathcal H}(T^\dagger)$ is standardly stratified. In fact,
 it has stratifying system consisting of all $\Delta(\omega):=\Hom_{\mathcal H}(S_\omega,T^\dagger)$, with $S_\omega$ ranging over the dual left cell modules and relative to the quasi-poset $(\Omega,\preceq)$. \end{thm}

\begin{proof} The result follows by applying Theorem \ref{thm2.5}, as modified by Remark
\ref{Rem2.6}, by using Lemma \ref{lem3.8}, taking $\mathscr K={\mathbb Z}[t,t^{-1}]$.  The projectivity of $\Delta(\lambda)$ over $\mathscr K$  can be proved by \cite[Cor. C.19]{DDPW08} and
the argument after its proof, which uses the Auslander-Goldman
Lemma (see \cite[Lem. C.17]{DDPW08}). Alternatively, see \cite[Cor. 1.2.12]{DPS98}.  
 The projective
$A^\dagger$-modules for (SS1) and (SS3) in (2.1) may be taken as the
various $\Hom_{\mathcal H}(X_\omega,T^\dagger)$. We leave the straightforward
details to the reader.\end{proof}

\begin{rem}(a) We mention, with only a brief indication of the proof, that $T^\dagger$ can be chosen with the
regular module $\mathcal H$ as a direct summand. We do not yet know if it is
possible to do the same with other permutation module analogs. In the case of the regular module $\sH$ itself, one constructs a $\mathscr K$-split injective composite
 $$\sH\to\oplus_j(\sH/\sH^{j})\to T^\dagger$$ 
 and uses the well-known fact that $\sH$ is self-dual as a left ($\mathscr K$-torsion free) $\sH$-module (thus, ``injective relative to $\mathscr K$'').

(b) The referee asked if, in the more general context of Theorem \ref{thm4.7}, one always gets a standardly stratified algebra from the construction above Theorem \ref{maintheorem}, letting $\mathscr S_i$ be generated by the single module $H^i/H^{i-1}$. This is true when
 $\mathscr K$ is regular of Krull dimension at most 2 (e.g., $\mathscr K={\mathbb Z}[t,t^{-1)}]$) if each $H^i/H^{i-1}$ is projective over $\mathscr K$ (a property of dual left cell modules
 implicitly 
 used in the proof of Theorem \ref{maintheorem} to ensure $\Delta(\la)$ is projective over $\mathscr K$).
 Another positive answer occurs if $\mathscr K$ is a DVR or Dedekind domain, without restriction on $H^i/H^{i-1}$. The latter module is always torsion-free, but, in the generality of Theorem 4.7, not much more is known about it.

(c) In the general Lusztig setup discussed above, after Remark \ref{4.8}, one knows $\sH^i/\sH^{i-1}$ is a direct sum of dual left cell modules with a largely explicit action of $\sH$ available using a (generalised) Kazhdan--Lusztig basis.
The height function, which determines $\sH^i/\sH^{i-1}$, is also important. Though not
 needed for our argument above,  an explicit choice may usually be given using Lusztig's $a$-function. For instance, the $a$-function may be used in the ``split'' or ``quasi-split case'' (in the terminology of Lusztig \cite{Lu03}). These cases include all unknown instances of our conjecture \cite{DPS98, DPS15} mentioned in the introduction.  
\end{rem}

\begin{appendix}
\section{A summary of exact categories}
This brief appendix summarizes, for the convenience of the reader, some basic material concerning exact
categories. We closely follow Keller's treatment in the appendix to  \cite{DRSS99}. (See also
Keller's paper \cite{K90}.)


Let $\mathscr A$ be an additive category. We do not repeat the standard definition, but refer to \cite[Chp. 9, \S1]{Mac94}
for a precise discussion.
A pair $(i,d)$ of composable morphisms $i:X\to Y$ and $d:Y\to Z$ in $\mathscr A$ is called exact if $i:X\to Y$ is the kernel of $d:Y\to Z$ and $d$ is the cokernel of $i$.
Let $\mathscr E$ be a class of exact pairs, which is closed under isomorphisms. If $(i,d)\in\mathscr E$, then
$i$ (resp., $d$) is called an inflation (resp., deflation), and the pair $(i,d)$ itself can be called a conflation. We often just write $X\overset{i}\to Y\overset{j}\to Z$ or merely $X\to Y\to Z$ to denote elements (i.e., conflations) in $\mathscr E$.

The pair $(\mathscr A,\mathscr E)$ is called an exact
category provided the following axioms hold:

\medskip
\noindent
 0. $1_0\in\Hom(0,0)$ is a deflation, where $0$ is the zero object in $\mathscr A$.

 \smallskip\noindent
1.  The composition of two deflations is a deflation.

\smallskip\noindent
2. Morphisms $Y\overset{d}\longrightarrow Z \overset{f}\longleftarrow Z'$ in $\mathscr A$ in which $d$ is a deflation can
be completed to a pullback diagram
$$\begin{CD} Y' @>d'>>Z'  \\
@Vf'VV @VfVV \\
Y @>d>> Z \end{CD}
$$
in which $d'$ is a deflation.

\smallskip\noindent
2$^\circ$. Morphisms $X'\overset{f}\longleftarrow X\overset{i}\longrightarrow Y$ in $\mathscr A$ in which $i$
is an inflation can be completed to a pushout diagram
$$\begin{CD} X @>i>> Y   \\
@VfVV @V{f'}VV \\
X' @>{i'}>> Y' \end{CD}
$$
in $\mathscr A$ in which $i'$ is an inflation.

\begin{rems}\label{RemDog} (a) The axioms above are part of Quillen's axioms \cite{Q73} for an exact category, and they are shown in \cite{K90} to be equivalent to the full set of axioms. Since the Quillen axioms are self-dual, it follows that any exact category
in the sense of the above conditions also satisfies each corresponding dual condition. For example, the composition of any two
inflations is an inflation.

(b)  Continuing the above remark, note that the opposite category ${\mathscr A}^\op$ inherits an exact category structure from that of $\mathscr A$. Now assume that
 $\mathscr A$ is small. (If one believes in the set-theoretic philosophy of universes, every ${\mathscr A}$ can be regarded as small in an appropriate set-theoretic universe.) Applying \cite[Prop. A2]{K90} to the opposite category $\mathscr A^\op$, we find that there is an abelian category $\mathscr B$ and faithful full embedding $G:{\mathscr A}\to\mathscr B$, such that
an exact pair $(i,d)$  belongs to $\mathscr E$ if and only if $0\to G(X)\overset{G(i)}\longrightarrow G(Y)\overset{G(d)}\longrightarrow G(Z)\to 0$ is a short exact sequence
in $\mathscr B$. Moreover, we can assume that the strict image $\mathscr M$ of $G$ (which is equivalent to $\mathscr A$) is closed under extensions in $\mathscr B$.

(c) Assume the setting of Axiom 2. Let $i:X\to Y$ (resp. $i':X'\to Y'$ ) be inflations with $(i,d)$ and $(i',d')$ in $\mathscr E$.
Then there is a commutative diagram
$$\begin{CD} X' @>i'>>Y' @>d'>>Z'  \\
@VVV @VVV @VVV \\
X @>i>> Y @>d>> Z .\end{CD}$$
The morphism $X'\to X$, induced
by the zero composition $X'\to Y'\to Y\to Z$, is an isomorphism, with inverse given by the map $X\to X'$ induced from the evident zero morphism $X\to Y'\to Z$ (where $X\to Y'$ is obtained by pull-back from the morphism $i:X\to Y$, and the zero morphism $X\to Z'$). This is all in $(\mathscr A, \mathscr E)$, but 
a similar construction may be made with a pullback of $G(Z')\to G(Z)$ and $G(d):G(Y)\to G(Z)$.
It follows easily that $G(Y')$ is a pullback in $\mathscr B$ of these two maps. Similar constructions apply for any exact embedding of $\mathscr A$ into an abelian category, or even an exact category; see \cite[Prop.~5.2]{Bu10}. A dual discussion applies for Axiom 2$^\circ$.


(d) The embedding in (b) can be used to prove ``with elements'' that useful exact sequences belong to
$\mathscr E$. For example, if we put $f':Y'\to Y$ and $f: Z'\to Z$, then there is a sequence $Y'\overset{\iota}\to Z'\oplus Y\overset{\epsilon}\to Z$ of maps in $\mathscr A$, where $\iota$ is defined as the ``product map'' associated with $d', f'$ and $\epsilon$ the coproduct map of $-d,f$. We claim $(\iota,\epsilon)$ belongs to $\mathscr E$.
The analogous assertion for an abelian module category is easy to prove. (Arguing with the given notation, a pullback of $f, d$ may be constructed in $Z'\oplus Y$ as all elements $z'\oplus y$ with $f(z')=d(y)$; this is precisely the kernel of $\epsilon$, and is naturally isomorphic to any other pullback, such as $Y'$.) We may assume $\mathscr B$ has been replaced by such a category.
 Applying $G$ gives a sequence
$G(Y')\overset{G(\iota)}\longrightarrow G(Z')\oplus G(Y)\overset{G(\epsilon)}\longrightarrow G(Z)$ in $\mathscr B$. One can see this sequence is exact by the facts that (1) $G(\iota)$
 is the product map of $G(f')$ and $G(d')$ and $G(\epsilon)$  is the coproduct map of $G(-d)$ and $G(f)$; (2) applying $G$ to the right square of the digram above yields a pullback diagram by (c).
Thus, $(\iota,\epsilon)$ is exact in $\mathscr A$, that is, $(\iota,\epsilon)$ belongs to $\mathscr E$.

(e) The abelian category $\mathscr B$ can also be used to extend the exact sequence
in Proposition \ref{prop4.3} below to the right by one term, as in the argument for Proposition \ref{prop4.4}. As previously mentioned, \cite{DRSS99} effectively gives a general 6 term version, using the ``split idempotent`` hypothesis, which we cannot assume.
 \end{rems}

Let $({\mathscr A}, {\mathscr E})$ be an exact category. For $X,Z\in\mathscr A$, let ${\mathscr E}(Z,X)$ be the set of sequences
$X\to Y\to Z$ in $\mathscr E$. Define the usual equivalence relation $\sim$ on ${\mathscr E}(Z,X)$ by
putting 
$$(X\to Y\to Z)\,\sim\,(X\to Y'\to Z)$$ provided there is a morphism $Y\to Y'$ giving a commutative
diagram
$$\begin{CD} X @>>> Y @>>> Z\\
@| @VVV @|\\
 X @>>> Y' @>>> Z
 \end{CD}
 $$
 The morphism $Y\to Y'$ is necessarily an isomorphism (as follows from Remark \ref{RemDog}(b) and a diagram chase, for
 example). Let $\Ext^1_{\mathscr E}(Z,X)={\mathscr E}(Z,X)/\sim$.

\begin{prop}\label{prop4.3} (a) $\Ext^1_{\mathscr E}(Z,X)$ has a natural abelian group structure such that given any $A\to B\to C$
in $\mathscr E$ and object $Z\in\mathscr A$, there are exact sequences
$$0\to\Hom_{\mathscr A}(Z, A)\longrightarrow \Hom_{\mathscr A}(Z,B)\longrightarrow \Hom_{\mathscr A}(Z,C)
\overset{f}\longrightarrow\Ext_{\mathscr E}^1(Z,A)$$
where $f$ is defined by pullback as in Axiom 2. A dual contravariant version holds, using the
contravariant functor $\Hom_{\mathscr A}(-,Z)$ and pullback as in Axiom 2$^\circ$.

(b)  Let $\mathscr K$ be a fixed
commutative, Noetheran ring. If $\mathscr A$ is a $\mathscr K$-category, then $\Ext^1_{\mathscr E}(Z,A)$ is naturally a ${\mathscr K}$-module.
\end{prop}

\begin{proof} The usual argument involving the Baer sum ($\alpha+\beta=\nabla_X(\alpha\oplus\beta)\Delta_Z$)  proves (a); see \cite[p. 85, (5.4)]{Mac94}.  We next prove (b).
Using standard embedding theorems, we can reduce to the case where $\mathscr A$ is a ${\mathscr K}$-category of ${\mathscr K}$-modules
(We remark that this is the only case to which we make applications in this paper.). Assuming this, we have what
appears to be two actions of ${\mathscr K}$ on $\Ext^1_{\mathscr E}(Z,X)$, one through the action of ${\mathscr K}$ on Z, and one through its action on X.
The first of the two actions uses a pullback of multiplication by any given element b in ${\mathscr K}$ on Z, and the second uses a pushout of the $X\to Y\to Z$
action of b on X.  We take part (b) as asserting, in this context, that the actions are the same, and that is (all of) what
we will prove.

Suppose we are given an element of $\Ext^1_{\mathscr E}(Z,X)$ represented by $X\overset{i}\to Y\overset{d}\to Z$, and
let $b \in{\mathscr K}$.  Form the pullback and pushout
objects as above, denoting the pullback by $Y'$ and the pushout by $Y^\#$.   The pullback object is formed by all pairs
$(y,z)$ with $dy=bz (y \in Y, z\in Z)$. It is an
object in $\mathscr A$ which is  a subobject of $Y\oplus Z$.  There is an evident sequence $X\to Y'\to Z$, which we also call a pullback. The pushout object
$Y^\#$ is formed as a quotient of $X\oplus Y$ by the subobject $W$ consisting of all pairs $(-bx,ix)$, with $x\in X$. We represent an element of this
quotient as a bracketed pair $[x,y]$, with the representative pair $(x,y)$ well-defined only up to addition of an element of $W$. There is a corresponding
pushout sequence $X\to Y^\#\to Z$.  We claim this sequence represents the same element of $\Ext^1_{\mathscr E}(Z,X)$ as the pullback sequence with  $Y'$. To prove this,
all we have to do is exhibit a map $Y^{\#}\to Y'$ in the ${\mathscr K}$-category ${\mathscr A}$ giving the expected commutative diagram. Such a map may be defined by sending a pair $x,y \in X\oplus Y$ to $(by+ix,dy)\in Y\oplus Z$, a pair which is actually in
$Y'$, since $d(by+ix) =b(dy)$. Moreover, the map has $W$ in its kernel since, if $x\in X$,
$(b(ix)+(-bx), d(ix))=(0,0)$. Thus,  induces a map to $Y^{\#}\to Y'$.  We leave it to the reader to check the required commutativites.  This proves the claim and completes the proof of part (b).\end{proof}

 For a relatively recent survey of exact categories, starting from the Quillen axioms (though without any explicit discussion of
 $\Ext^1_{\mathscr E}$), see \cite{Bu10}.

\section{Idempotent ideals}

The following result is proved in \cite{CPS90}. For convenience, we indicate a short proof.

\begin{prop}\label{prop5.1} Let $J$ be an idempotent ideal in a ring $A$. Assume that $_AJ$ is projective. Let
$M,N$ be $A/J$-modules. For any integer $n\geq 0$, inflation provides an isomorphism
$$ \Ext^n_{A/J}(M,N)\overset\sim\longrightarrow\Ext^n_A(M,N)$$
of abelian groups. (On the right hand side, $M,N$ are regarded as $A$-modules through the morphism
$A\to A/J$.)\end{prop}

\begin{proof} Using the short exact sequence $0\to J\to A\to A/J\to 0$ of left $A$-modules, the projectivity
of $_AJ$ implies that $\Ext^n_A(A/J,N)=0$ for $n>1$. Since $J^2=J$, $\Hom_A(J,N)=0$. Thus, any
projective $A/J$-module is acyclic for the functor $\Hom_A(-,N)$. The proposition follows.
\end{proof}

\end{appendix}


\begin{thebibliography}{W}


\bibitem[\sf Bu10]{Bu10} T. B\"uhler, Exact Categories, {\it Expo. Math.} {\bf 28} (2010), 1--69.

\bibitem[\sf BH61]{BH61} M.C.R. Butler and G. Horrocks,
Classes of extensions and resolutions, {\it Philos. Trans. Roy. Soc.
} {\bf 254} (1962), 155--222.

\bibitem[\sf CPS90]{CPS90} E. Cline, B. Parshall, and L. Scott, Integral and graded quasi-hereditary algebras, I, {\it
J. Algebra} {\bf 131} (1990), 126--160.


\bibitem[\sf DDPW08]{DDPW08} B. Deng, J. Du, B. Parshall, J-P. Wang, {\it Finite dimensional algebras and quantum
groups,} Math. Surveys and Monographs, {\bf 150}, Amer. Math. Soc., Providence, (2008).

\bibitem[\sf DRSS99]{DRSS99} P. Dr\"axler, I. Reiten, S. Smal\/o, and \/O. Solberg, Exact categories and vector space
categories, {\it Trans. A.M.S.} {\bf 351} (1999), 647--682 (with an appendix by B. Keller).


\bibitem[\sf  DPS98]{DPS98} J. Du, B. Parshall, and L. Scott, Stratifying endomorphism algebras associated to
Hecke algebras, {\it J. Algebra} {\bf 203} (1998), 169--210.


\bibitem[\sf DPS15]{DPS15} J. Du, B. Parshall, and L. Scott,  Extending Hecke endomorphism algebras, {\it Pacific
J. Math.}, {\bf 279-1} (2015), 229--254 (special issue in memory of Robert Steinberg).



\bibitem[\sf GR97]{GR97} P. Gabriel and A. Roiter, {\it Representations of Finite-Dimensional Algebras,} Springer
(1997).

\bibitem[\sf GGOR03]{GGOR03} V. Ginzburg, N. Guay, E. Opdam, and R. Rouquier, On the category $\sO$ for rational
Cherednik algebras, {\it Invent. math.} {\bf 156} (2003), 617--651.


\bibitem[\sf K90]{K90} B. Keller, Chain complexes and stable categories, {\it Manus. Math.} {\bf 67} (1990), 379--417.

\bibitem[\sf K96]{K96} B. Keller, Derived categories and their uses, in: Handbook of Algebra I, (1996), 671--701.


\bibitem[\sf Lu03]{Lu03} G. Lusztig, {\it Hecke algebras with unequal parameters,} CRM Monograph Series {\bf 18},
Amer. Math. Soc. Providence, (2003).

\bibitem[\sf Mac94]{Mac94} S. MacLane, {\it Homology}, Springer (1994).


\bibitem[\sf Ne90]{Ne90} A. Neeman, Derived category of an exact category, {\it J. Algebra} {\bf 135} (1990), 388-394.


\bibitem[\sf Q73]{Q73}  D. Quillen, {\it Higher algebraic K-theory} I, Springer LNM {\bf 341} (1973), 85--147.





\end{thebibliography}
\end{document}